\documentclass[11pt]{article}
\usepackage{graphicx}
\usepackage{amsfonts}
\usepackage{amsmath}
\usepackage{amssymb,latexsym,times}

\def\ef{\mbox{EF}}

\def\MEG{\mbox{MEG}}
\def\ma{M}
\def\mb{N}

\usepackage{amssymb,amsthm,amsmath}
\usepackage{latexsym}
\usepackage{hyperref}
\usepackage{endnotes}
\let\footnote=\endnote
\newtheorem{theorem}{Theorem}

\let\footnote=\endnote
\newcommand{\open}{\Bbb}

\newcommand{\oN}{{\open N}}

\def\fraisse{Fra\"\i ss\' e\ }

\begin{document}

\title{The Strategic Balance of Games in Logic\footnote{This paper is written in response to a question by Samson Abramsky, about where in my book \cite{MR2768176} are the promised translations between strategies of the various games that I present. To my surprise I had not given the translations explicitly in the book, perhaps because in the book I move freely from games to inductive definitions and back. In this paper the translations are given explicitly and I thank Samson for pointing out the omission in my book.}}
\author{Jouko V\"a\"an\"anen\thanks{This project has received funding from the Academy of Finland (grant No 322795) and from the European Research Council (ERC) under the
European Union’s Horizon 2020 research and innovation programme (grant agreement No
101020762).}\\ University of Helsinki, Finland\\ and\\ University of Amsterdam, The Netherlands}
\maketitle

\begin{abstract}
Truth, consistency and elementary equivalence can all be characterised in terms of games, namely the so-called evaluation game, the model-existence game, and the Ehrenfeucht-\fraisse game. We point out the great affinity of these games to each other and call this phenomenon the \emph{strategic balance in logic}. In particular, we give explicit translations of strategies from one game to another.
\end{abstract}

\section{The three games of logic}

The Game Theoretical Semantics of first order logic, and many other logics too, is a combination of three games: The Evaluation Game, the Model Existence Game and the EF game (short for Ehrenfeucht-Fra\"\i ss\'e Game). These games are closely related to each other and cover the main semantic concepts of logic. We use the term ``Game Theoretical Semantics" for the general approach that emphasises and utilizes these games in semantics, rather than some other approach, e.g. set-theoretical semantics in the Tarskian sense. 

In the Evaluation Game we have a structure $A$ and a sentence $\phi$. The game gives meaning to the proposition that $\phi$ is true in $A$. The meaning is equivalent to the one given by Tarski's Turth Definition. In the Model Existence Game the structure $A$ is missing and we have just the sentence $\phi$. The game gives meaning to the proposition that $\phi$ is consistent, or alternatively, that there is a structure in which $\phi$ is true. This game appears in different incarnations throughout logic, all the way from Gentzen's Natural Deduction, to Beth Tableaux, to Dialogical Games, to Hintikka's model sets, to semantic trees, to Smullyan's  consistency properties. Finally, in the EF game the sentence $\phi$ is missing and we just have two models $A$ and $B$. The game gives meaning to the proposition that some sentence is true in $A$ but false in $B$. Alternatively we can say that the EF game gives meaning to, or equivalently a definition of the elementary equivalence of $A$ and $B$. 

\begin{figure}
\begin{center}
\includegraphics[width=0.8\textwidth]{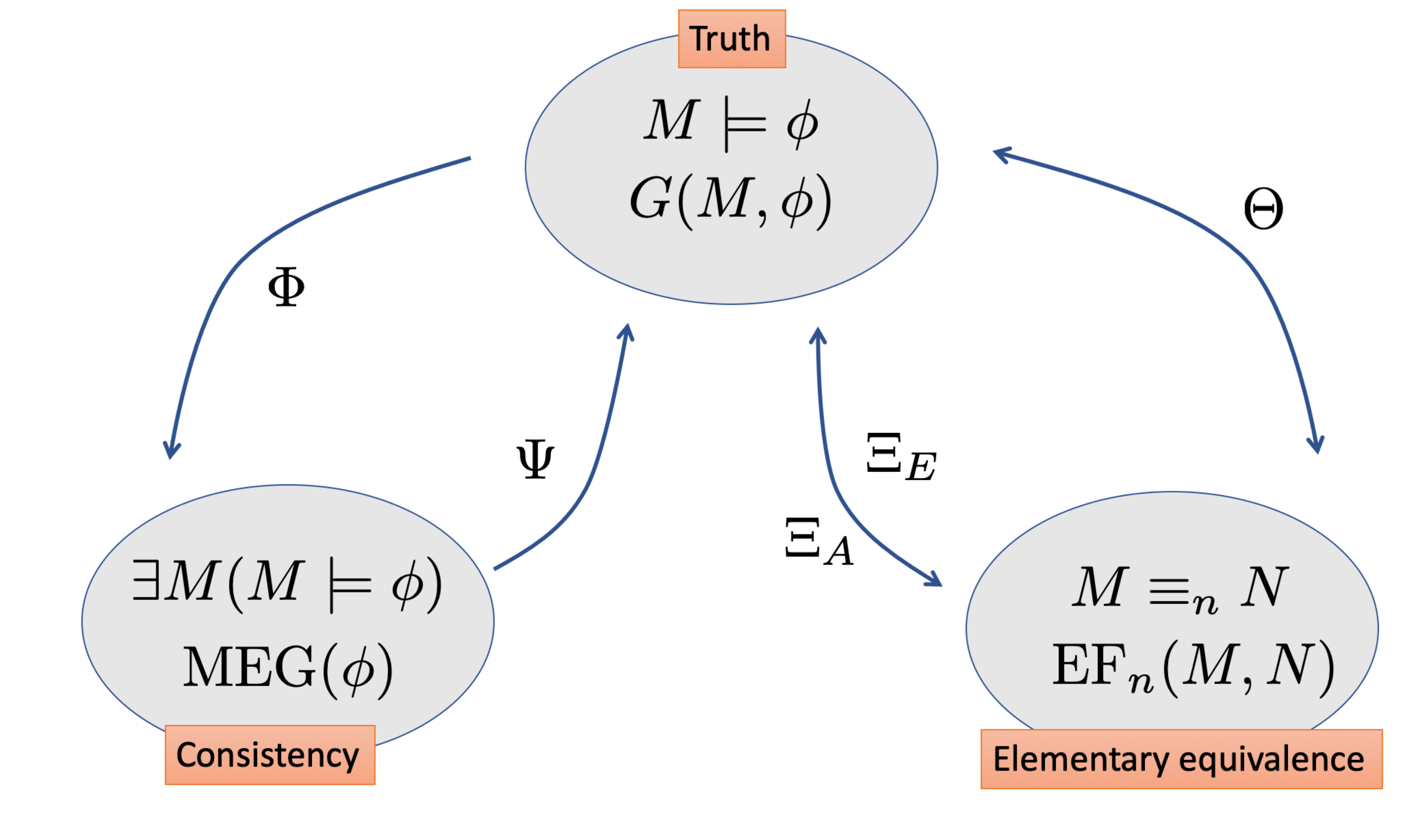}
\end{center}
\caption{The translations of strategies.\label{fig}}\end{figure}

The meanings that these games give to truth, consistency, and elementary equivalence are intertranslatable  in the sense that a strategy in one can be converted to a strategy in the other (see Figure~\ref{fig}). If the former is a winning strategy, so is the latter. We need not traverse via Tarski's Truth Definition. In this sense  Game Theoretical Semantics stands autonomously on its own feet. This is important for those who find Tarski's approach too set-theoretical. However, the set-theoretical treatment is arguably the best representation of even the Game Theoretical Semantics, despite the intuitive appeal of the latter, when one goes into the gritty details of complicated proofs.

The three games, Evaluation Game, Model Existence Game, and EF game, penetrate into the heart of logic where we find the concepts of truth, consistency, provability, definability and isomorphism. Their mathematical interpretations and elaborations lead us into the deepest questions of mathematical logic, especially model theory and set theory. They provide tools for understanding central questions of philosophical logic such as intuitionistic, modal and other non-classical logic. They have become indispensable in theoretical computer science modelling interaction and complexity.

For more details and for references to the original sources we refer to \cite{MR2768176}. Everything in this paper was essentially known already in the fifties and sixties. What is perhaps new is the emphasis on the mutual relationships between the three games, taking the form of a collection of translations between strategies. This aspect has been forcefully emphasised  in \cite{MR3236216}. The author feels that despite the ripe age of the main results, this topic has more to offer than has been discovered so far. An interesting recent development is \cite{DBLP:conf/lics/AbramskyDW17}.

\medskip

\noindent{\bf Notation: } If $X$ is a set, an $X$-assignment is a mapping with a set of variables as its domain and values in $X$. If $s$ is an $X$-assignment and $a\in X$, then $s(a/x)$ is the $X$-assignment which agrees with $s$ except that it gives $x$ the value $a$. If $X$ is clear from the context, we drop it  and talk about assignments only.

\section{Given both a sentence and a model: the Evaluation Game}

We describe a game which gives meaning to the \emph{truth} of a  given sentence in a given model. In its simplest form this game, the Evaluation Game, is the following: Suppose $M$ is a structure for a  vocabulary $L$, $\phi$ is a first order (or, more generally, infinitary) $L$-sentence  and $s$ is an assignment. We allow identity.  Let $\Gamma(s)$ denote the set of all literals i.e. atomic and negated atomic formulas that $s$ satisfies in $M$. As the logical operations allowed in $\phi$ we include $\neg,\wedge,\vee,\forall$ and $\exists$. To make the situation as clear as possible, we assume that $\phi$ is  in  Negation Normal Form, meaning that negation occurs only in front of atomic formulas. The game $G(M,\phi)$ has two players Abelard and Eloise. Intuitively, Eloise defends the proposition that $\phi$ is 
(informally) true in $M$ and Abelard doubts it. All through the game the players are inspecting an $L$-sentence and an assignment $s$. We call the pair $(\psi,s)$ a \emph{position} of the game. In the beginning of the game the position is $(\phi,\emptyset)$. 
The rules are as follows: Suppose the position is $(\psi,s)$.
\begin{enumerate}
\item[(1)] If $\psi$ is a literal, the game ends and Eloise wins if $\psi\in\Gamma(s)$. Otherwise Abelard wins.
\item[(2)] If $\psi$ is $\psi_0\wedge\psi_1$, then Abelard chooses whether the next position is
$(\psi_0,s)$ or $(\psi_1,s)$.
\item[(3)] If $\psi$ is $\psi_0\vee\psi_1$, then Eloise chooses whether the next position is
$(\psi_0,s)$ or $(\psi_1,s)$.
\item[(4)] If $\psi$ is $\forall x\theta$, then Abelard chooses $a\in M$ and the next position is
$(\theta,s(a/x))$.
\item[(5)] If $\psi$ is $\exists x\theta$, then Eloise chooses $a\in M$ and the next position is
$(\theta,s(a/x))$.
\end{enumerate}

 A \emph{strategy} for a player is a set of rules (functions) that describe exactly how that player should choose, depending on how the two players have played at earlier moves. It is a \emph{winning strategy} if the player wins whichever way the opponent plays. The game $G(M,\phi)$ is clearly a determined game i.e. always one of the players has a winning strategy. 
 
 We say that $\phi$ is \emph{true in} $M$ if Eloise has a winning strategy in $G(M,\phi)$. This is the game-theoretical meaning of truth in a model. We can go further and say that the \emph{game}  $G(M,\phi)$  is the \emph{meaning} of $\phi$ in $M$. Here meaning would be a broader concept than the mere truth or falsity of $\phi$. 

The first to formulate the Evaluation Game explicitly was probably J. Hintikka \cite{Hintikka1968-HINLFQ} who later advocated the importance and usefulness of the game forcefully. Hintikka mentions L. Wittgenstein \cite{MR0078292} as an inspiration. Earlier L. Henkin \cite{MR0143691}  formulated  a game theoretic approach to quantifiers pointing out generalizations to very general infinitary and partially ordered quantifiers. Nowadays the Evaluation Game is a standard tool in mathematical, philosophical and computer science logic. 
 
 The game $G(M,\phi)$ reflects the entire syntactical structure of $\phi$ in the sense that the game $G(M,\phi\wedge\psi)$ is intimately related to the two games $G(M,\phi)$ and $G(M,\psi)$, the same with $G(M,\phi\vee\psi)$, $G(M,\exists x\phi)$ and $G(M,\forall x\phi)$. This phenomenon is a manifestation of the broader concept of \emph{compositionality}.
 
If $\phi$ is propositional i.e. has only zero-place relation symbols and no constant or function symbols, and no quantifiers, then only moves (1)-(3) occur in $G(M,\phi)$, and the assignments can be forgotten.  If $\phi$ is universal, the game $G(M,\phi)$ has no moves of type (5). If it is existential, the game has no moves of type (4). If universal-existential, then all type (5) moves come before type (4) moves.
 Furthermore, if we add new logical operations to our logic, such as infinite conjunctions and disjunctions, generalized quantifiers or higher order quantifiers, it is clear how to modify the game $G(M,\phi)$ to accommodate the new logical operations. For example, for $\phi$ in $L_{\infty\omega}$, we modify above (2) and (3) as follows:
 \begin{enumerate}
\item[(2')] If $\psi$ is $\bigwedge_{i\in I}\psi_i$, then Abelard chooses $i\in I$ and   the next position is
$(\psi_i,s)$.
\item[(3')] If $\psi$ is $\bigvee_{i\in I}\psi_i$, then Eloise chooses $i\in I$ and   the next position is
$(\psi_i,s)$.
\end{enumerate}
This extends also to  so-called team semantics \cite{MR2351449}, in which case a position in the game $G(M,\phi)$ is a pair $(\psi,X)$, where $\psi$ is a subformula of $\phi$ and $X$ is a team. Finally, if $M$ is a Kripke-model and $\phi$ a sentence of modal logic, the game $G(M,\phi)$ is entirely similar. The assignments have a singleton domain $\{x_0\}$ and values in the frame of $M$. The moves corresponding to $\Diamond$ and $\Box$ are like (4) and (5):
\begin{enumerate}
\item[(4')] If $\psi$ is $\Box\theta$, then Abelard chooses a node $b$ accessible from $s(x_0)$ and the next position is
$(\theta,s(b/x_0))$.
\item[(5')] If $\psi$ is $\Diamond \theta$, then Eloise chooses a node $b$ accessible from $s(x_0)$ and the next position is
$(\theta,s(b/x_0))$.
\end{enumerate}

The game $G(M,\phi)$ can also reflect the structure of $M$. In fact, the games $G(M\times N,\phi)$, $G(M+N,\phi)$, and $G(\Pi_iM_i/F,\phi)$ are intimately related to the games $G(M,\phi)$, $G(N,\phi)$ and $G(M_i,\phi)$ \cite{MR50:12646}. 
The game $G(M,\phi)$ is also helpful in finding a countable submodel $N$ of $M$ with desired properties. For any strategy $\tau$ of Eloise in $G(M,\phi)$ let $T(M,\tau)$ be the set of countable submodels $N$ of $M$ such that $N$ is closed under the functions of $M$ 
and under $\tau$ i.e. if Abelard plays in (4) always $a\in N$, then also Eloise plays in (5) always $b\in N$. 
Note that if $N\in T(M,\tau)$, then $\tau$ is necessarily  a strategy of Eloise in $G(N,\phi)$. Moreover, if $\tau$ is a winning strategy in $G(M,\phi)$, then it is also a winning strategy in $G(N,\phi)$. The {\bf L\"owenheim-Skolem Theorem} of $L_{\omega_1\omega}$ is essentially the statement that $T(M,\phi)\ne\emptyset$, when $\phi\in L_{\omega_1\omega}$. In the \emph{Cub Game} due to D. Kueker \cite{MR56:15406} there are two players Abelard and Eloise, a set $X$ and a set $T$ of countable subsets of $X$. During the game, which we denote $G_T$, the players choose  alternatingly  elements $a_n\in X$, Abelard being the first to move. After $\omega$ moves a set $\{a_0,a_1,\ldots\}$ has been produced. We say that Eloise is the winner if this set is in $T$, otherwise Abelard is the winner. This game need not be determined. If Eloise has a winning strategy in $G_T$, then $T$ contains a so-called cub (closed unbounded) set of countable subsets of $X$.

Below is a strong form of a L\"owenheim-Skolem Theorem for  $L_{\omega_1\omega}$, based on the Evaluation Game. It shows that the countable submodels for which the given strategy works for Eloise are ``everywhere" (i.e. cub) inside $M$.  The theorem is due to \cite{MR56:15406}.

\begin{theorem}\label{ls}
Suppose $M$ is a model in a countable vocabulary  and $\phi$ is a sentence of $L_{\omega_1\omega}$. There is a mapping $\Upsilon$ such that if $\tau$ is a strategy of Eloise in $G(M,\phi)$, then $\Upsilon(\tau)$ is a strategy of Eloise in $G_{T(M,\phi)}$. If $\tau$ is a winning strategy, then so is $\Upsilon(\tau)$.
\end{theorem}

\begin{proof}(Sketch)
Using bookkeeping Eloise makes sure during the game $G_{T(M,\phi)}$ that her moves $a_{2n+1}\in M$ render the final set $\{a_0,a_1,\ldots\}$ such that it is both the domain of a submodel of $M$ and if Abelard plays his moves of type (4) in $\{a_0,a_1,\ldots\}$ then so does Eloise, using $\tau$, in her moves of type (5). 
\end{proof}
In conclusion, the game $G(M,\phi)$ is a versatile tool for understanding the meaning of a logical sentence $\phi$ in a mathematical structure $M$. Alternative tools, yielding equivalent results, are inductive methods such as the so-called Tarski's Truth Definition and the theory of inductive definitions.

\section{The model is missing: the Model Existence Game}

Here we have a sentence and we ask whether the sentence has a model. Thus this is about \emph{consistency} and its opposite, \emph{contradiction}. In particular, the issue is whether there is some model $M$ such that Eloise can win $G(M,\phi)$? This question is governed by the Model Existence Game $\MEG(\phi)$. 

Suppose  $\phi$ is a first order sentence. As the logical operations allowed in $\phi$ we include $\neg,\wedge,\vee,\forall$ and $\exists$. To make this as simple as possible, we assume again that $\phi$ is  in  Negation Normal Form. The game $\MEG(\phi)$ has two players Abelard and Eloise. Intuitively, Eloise defends the proposition that $\phi$ has a model and Abelard doubts it. Abelard expresses his doubt by asking questions. It may be the case that Eloise has a model at her disposal but she does not want to reveal it, apart from using it to give answers to Abelard's questions, but she may also just pretend that she does. For all we know there may not be any model. Conceivably $\phi$ is a contradiction, but the game should settle this. We let $C=\{c_0,c_1,\ldots,c_n,\ldots\}$ be a set of new distinct constant symbols. Intuitively these are names of elements of the supposed model.  A position of the game is a finite collection $S$ of pairs $(\psi,s)$, where $\psi$ is a subformula of $\phi$ and $s$ is a $C$-assignment. In infinitary logic we may consider also infinite $S$. In the beginning of the game the position is $\{(\phi,\emptyset)\}$. The rules are as follows: Suppose the position is $S$.
\begin{enumerate}
\item If $(\psi_0\wedge\psi_1,s)\in S $, then Abelard may decide that the next position is
$S\cup\{(\psi_0,s)\}$ or $S\cup\{(\psi_1,s)\}$.
\item If $(\psi_0\vee\psi_1,s)\in S$, then Abelard may decide that Eloise has to choose whether the next position is
$S\cup\{(\psi_0,s)\}$ or $S\cup\{(\psi_1,s)\}$.
\item If $(\forall x\theta,s)\in S$, then Abelard may choose $n\in\oN$ and decide that the next position is
$S\cup\{(\theta,s(c_n/x))\}$.
\item If $(\exists x\theta,s)\in S$, then Abelard may decide that Eloise has to choose $c_n$ and the next position is
$S\cup\{(\theta,s(c_n/x))\}$.
\end{enumerate}
Abelard wins if at some point the position $S$ contains both $(\psi,s)$ and $(\neg\psi,s')$, where $\psi$ is an atomic formula and $s(x)=s'(x)$ for all variables $x$ in $\psi$. If Abelard does not win, then the game may continue for infinitely many moves and then Eloise is declared the winner.
Note that Abelard has a lot of control on how the game proceeds as he may focus attention on any formula in the finite set $s$. This is clearly a determined game.

The Model Existence Game is a game-theoretic rendering of the method of Beth Tableaux, with background in Gentzen's natural deduction. It was presented roughly as above, but without explicit mention of a game, for first order logic at about the same time by E. Beth \cite{MR0089151} and J. Hintikka \cite{MR69779}, later in a more abstract form by R. Smullyan \cite{MR0152430}, and for infinitary logic  by M. Makkai \cite{MR0255383}.

The following theorem, essentially due to \cite{MR0089151}, demonstrates how the Model Existence Game and the Evaluation Game can be combined in the somewhat trivial case that a model is not missing but actually known (see Figure~\ref{fig4}).

\begin{figure}
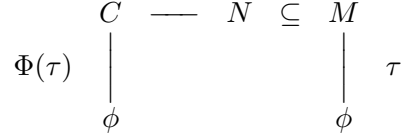

$$\begin{array}{ccccccc}
&C&-\!\!\!-\!\!\!-&N&\subseteq& M&\\
\Phi(\tau)&\bigg|&&&&\bigg|&\tau\\
&\phi&&&&\phi
\end{array}$$
\caption{From model to model existence\label{fig4}.}\end{figure}

\begin{theorem}Suppose the vocabulary of $M$ is countable. There is a mapping $\Phi$  so that if $\tau$ is a strategy of Eloise in $G(M,\phi)$, then $\Phi(\tau)$ is a  strategy of Eloise in $\MEG(\phi)$. 
If $\tau$ is a winning strategy, then so is $\Phi(\tau)$.\end{theorem}

\begin{proof}By Theorem~\ref{ls} there is a countable submodel $N$ of $M$ such that $\tau$ is a strategy of Eloise in $G(N,\phi)$ and if $\tau$ is a winning strategy in $G(M,\phi)$, then it is also a winning strategy in $G(N,\phi)$. Let $\pi:C\to N$ be an onto map. 
Let us say that a pair $(\psi,s)$ is a $\tau$-position if there is there is some sequence of positions  in $G(N,\phi)$, following the rules of $G(N,\phi)$ starting with $(\phi,\emptyset)$, Eloise using $\tau$, which ends at $(\psi,s)$. A \emph{$C$-translation} of the $\tau$-position $(\psi,s)$ is a pair $(\psi,s')$ where $s'$ is a $C$-assignment with 
$\pi(s'(x))=s(x)$. The strategy $\Phi(\tau)$ of Eloise in $\MEG(\phi)$ is to make sure that at all times the position $S$ consists only of $C$-translations of $\tau$-positions.  Let us now check that Eloise can actually follow this strategy: Suppose the position is $S$.
\begin{enumerate}

\item If $(\psi_0\wedge\psi_1,s')\in S $, then Abelard may decide that the next position is
$S\cup\{\psi_0,s'\}$ or $S\cup\{\psi_1,s'\}$. Let us assume he chooses $S\cup\{\psi_0,s')\}$ as the next position. There is a sequence $\alpha$ of moves  in $G(N,\phi)$ in which Eloise uses  $\tau$, which ends at $(\psi_0\wedge\psi_1,s)$, a $C$-translation of which is $(\psi_0\wedge\psi_1,s')$. We continue $\alpha$ by letting Abelard move $\psi_0$. Now $S\cup\{(\psi_0,s')\}$ consists of $C$-translations of $\tau$-positions. 

\item If $(\psi_0\vee\psi_1,s')\in S $, then Eloise can  decide that the next position is
$S\cup\{(\psi_0,s')\}$ or $S\cup\{(\psi_1,s')\}$.  There is a sequence $\alpha$ of moves  in $G(N,\phi)$ in which Eloise uses  $\tau$, which ends at $(\psi_0\vee\psi_1,s)$, a $C$-translation of which is $(\psi_0\vee\psi_1,s')$. We continue $\alpha$ by letting Eloise move whatever $\tau$ tells her to move, say $\psi_0$. Now $S\cup\{(\psi_0,s')\}$ consists of $C$-translations of $\tau$-moves. So we let Eloise choose $S\cup\{(\psi_0,s')\}$ as the next position.

\item If $(\forall x\theta,s')\in S $, then Abelard may choose $c\in C$ and   the next position is
$S\cup\{(\theta, s'(c/x))\}$. There is a sequence $\alpha$ of moves  in $G(N,\phi)$ in which Eloise uses  $\tau$, which ends at $(\forall x\theta,s)$, a $C$-translation of which is $(\forall x\theta,s')$. We continue $\alpha$ by letting Abelard move $(\theta,s(\pi(c)/x))$. Now $S\cup\{(\theta,s'(c/x))\}$ consists of $C$-translations of $\tau$-moves. 

\item If $(\exists x\theta,s)\in S $, then Eloise may choose $a\in M$ and  the next position is
$S\cup\{\theta\langle s(a/x)\rangle\}$.  There is a sequence $\alpha$ of moves  in $G(N,\phi)$ in which Eloise uses  $\tau$, which ends at $(\exists x\theta,s)$, a $C$-translation of which is $(\exists x\theta,s')$. We continue $\alpha$ by letting Eloise use $\tau$ to move $(\theta,s(a/x))$. Let $a=\pi(c)$. Now $S\cup\{\theta\langle s'(c/x)\rangle\}$ consists of $C$-translations of $\tau$-moves. We let Eloise choose $S\cup\{\theta\langle s'(c/x)\rangle\}$ as the next position.
\end{enumerate}

If Eloise follows this strategy and $\tau$ happens to be a winning strategy, then she wins because if at some point there are $(\psi,s')$ and $(\neg\psi,t')$ in the set $S$, such that $\psi$ is atomic and $s'(x)=t'(x)$ for variables $x$ in $\psi$, then these would be $C$-translations of $(\psi,s)$ and $(\neg\psi,t)$, where $s(x)=t(x)$ for $x$ in $\psi$, respectively, and $s$ (and $t$) would satisfy both $\psi$ and $\neg\psi$ in $N$ because $\tau$ is a winning strategy, a contradiction.
\end{proof}

The following theorem is more interesting because here we do not have a model to start with but we have to construct a model. In this sense this is reminiscent of G\"odel's Completeness Theorem. The theorem demonstrates how the Model Existence Game can yield a model which then has a good fit with the Evaluation Game (see Figure~\ref{fig5}).
The theorem goes back  to \cite{MR0089151}.

\begin{figure}
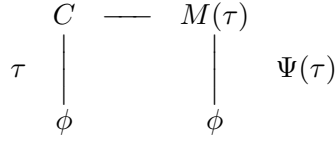

$$\begin{array}{rcccl}
&C&-\!\!\!-\!\!\!-& M(\tau)&\\
\tau&\bigg|&&\bigg|&\Psi(\tau)\\
&\phi&&\phi
\end{array}$$
\caption{From model existence to a model.\label{fig5}}\end{figure}

\begin{theorem} There are a mapping $\Psi$ and   an ``enumeration strategy" $\sigma_0$ of Abelard in $\MEG(\phi)$ such that  if $\tau$ is any  strategy of Eloise in $\MEG(\phi)$,  playing $\tau$ against $\sigma_0$ determines a unique model $M(\tau)$ and a strategy $\Psi(\tau)$ of Eloise in $G(M(\tau),\phi)$.  If $\tau$ is winning, then so is $\Psi(\tau)$. 
\end{theorem}

\begin{proof}
 To make the presentation a little easier, we assume $\phi$ has a relational vocabulary and does not contain the identity symbol. Let $\sigma_0$ be the following  ``enumeration strategy" of Abelard in $\MEG(\phi)$: During the game Abelard makes sure by means of careful bookkeeping that if $S$ is the position, then:
\begin{enumerate}
\item If $(\psi_0\wedge\psi_1,s)\in S $, then during the game he will at some position $S'\supseteq S$ decide that the next position is
$S'\cup\{(\psi_0,s)\}$ and at some further position $S''\supseteq S'$ he will decide that the next position is
$S''\cup\{(\psi_1,s)\}$.
\item If $(\psi_0\vee\psi_1,s)\in S$, then at some position $S'\supseteq S$ Abelard asks Eloise  to choose whether the next position is
$S'\cup\{(\psi_0,s)\}$ or $S'\cup\{(\psi_1,s)\}$.
\item If $(\forall x\theta,s)\in S$, then  for all $n$ during the game he will at some position $S'\supseteq S$ decide that the next position is
$S'\cup\{(\theta,s(c_n/x)\}$.
\item If $(\exists x\theta,s)\in S$, then at some position $S'\supseteq S$ Abelard will ask Eloise to choose $n$ after which the next position is 
$S'\cup\{(\theta,s(c_n/x))\}$.
\end{enumerate}
Let us play $\MEG(\phi)$ while Abelard uses this strategy and Eloise plays $\tau$. 
 Let $$\mathcal{S}=\langle S_n: n<\omega\rangle$$ be the infinite sequence of positions during this play. Since the strategies of the players have been fixed, the sequence $\mathcal{S}$ is uniquely determined. Note that  $S_n\subseteq S_{n+1}$ for all $n$. Let $\Gamma$ be the union of all the positions in $\mathcal{S}$.
We build a model $M=M(\tau)$ as follows: The domain of the model is $\{c_n : n\in\oN\}$. If $R$ is a relation symbol, then we let $R(c_{n_0},\ldots,c_{n_k})$ hold in $M$  if $(R(x_{n_0},\ldots,x_{n_k}),s)\in\Gamma$ for some $s$ such that  $s(x_i)=c_i$ for $i=n_0,\ldots,n_k$. The  strategy $\Psi(\tau)$ of Eloise in $G(M,\phi)$ is the following: She makes sure that if the position in $G(M,\phi)$ is $(\psi,s)$, then $(\psi,s)\in \Gamma$.  Let us see that she can follow the strategy throughout the game: 
\begin{enumerate}
\item If $\psi$ is $\psi_0\wedge\psi_1$, then Abelard chooses whether the next position is
$(\psi_0,s)$ or $(\psi_1,s)$. Suppose he chooses $(\psi_1,s)$. Now, $(\psi_0\wedge\psi_1,s)$ occurs in a position $S$ during the game, because $(\psi_0\wedge\psi_1,s)\in \Gamma$. According to how $\sigma_0$ is defined, Abelard has at some later position $S'\supseteq S$ of the game decided that the next position is $S'\cup\{(\psi_1,s)\}$. Hence $(\psi_1,s)\in\Gamma$.

\item If $\psi$ is $\psi_0\vee\psi_1$, then Eloise should choose whether the next position is
$(\psi_0,s)$ or $(\psi_1,s)$. We know $(\psi_0\vee\psi_1,s)\in S$ for some position $S$ during the game, because $(\psi_0\vee\psi_1,s)\in\Gamma$. By how $\sigma_0$ was defined, Abelard has at some later position $S'\supseteq S$  asked Eloise to choose between $(\psi_0,s)$ and $(\psi_1,s)$. The strategy $\tau$ has directed Eloise to choose, say $(\psi_0,s)$. Thus $(\psi_0,s)\in\Gamma$ and she can safely play $(\psi_0,s)$ in $G(M,\phi)$.

\item If $\psi$ is $\forall x\theta$, then Abelard chooses $a\in M$ and the next position is
$(\theta,s(a/x))$. Suppose he chooses $a=c_n$. Now, $(\forall x\theta,s)$ occurs in a position $S$ during the game, because $(\forall x\theta,s)\in \Gamma$. According to how $\sigma_0$ is defined, Abelard has at some later position $S'\supseteq S$ of the game decided that the next position is $S'\cup\{(\theta,s(c_n/x))\}$. 
Hence $(\theta,s(a/x))\in \Gamma$.

\item If $\psi$ is $\exists x\theta$, then Eloise should choose for which $n$ the next position is
$(\theta,s(c_n/x))$. We know $(\exists x\theta,s)\in S$ for some position $S$ during the game, because $(\exists x\theta,s)\in\Gamma$. By how $\sigma_0$ was defined, Abelard has at some later position $S'\supseteq S$  asked Eloise to choose $n$ for which the next position would be $S'\cup\{(\theta,s(c_n/x))\}$. The strategy $\tau$ has directed Eloise to indeed choose an $n$ leading to the new position $S'\cup\{(\theta,s(c_n/x))\}$. Thus $(\theta,s(c_n/x))\in\Gamma$ and she can safely play $(\theta,s(c_n/x))$ in $G(M,\phi)$.
\end{enumerate}

Suppose now that $\tau$ is a winning strategy of Eloise. We show that then $\Psi(\tau)$ is a winning strategy as well. Suppose the game $G(M,\phi)$ ends at a position $(\psi,s)$ where $\psi$ is atomic. Then $(\psi,s)$ is in $\Gamma$ whence $s$ satisfies $\psi$ in $M$.  Suppose, on the other hand, the game $G(M,\phi)$ ends at a position $(\neg\psi,s)$ where $\psi$ is atomic. Then $(\neg\psi,s)$ is in $S_n$ for some $n$. It suffices to show that $s$ does not satisfy $\psi$  in $M$. Suppose it does. Then there is $(\psi,s')\in S_m$ for some $m$ and some $s'$ such that $s(x)=s'(x)$ for variables $x$ in $\psi$. Then $S_{n+m}$ has both $(\neg\psi,s)$ and $(\psi,s')$, contrary to the assumption that $\tau$ is a winning strategy of Eloise.
\end{proof}

Let us suppose we investigate $\phi$ using the Model Existence Game and find out that Abelard has a winning strategy in $\MEG(\phi)$. We can form a tree, a Beth Tableaux, of all the positions when Abelard plays his winning strategy and we stop playing as soon as Abelard has won. Every branch of the tree is finite and ends in a position which includes a contradiction. We can make the tree finite by looking at the game more carefully. We can then view this tree as a kind of proof of $\neg\phi$. In this sense the Model Existence Game builds a bridge between proof theory and model theory. Strategies of Abelard direct us to proof theory, while strategies of Eloise direct us to model theory.

A winning strategy of Eloise in $\MEG(\phi)$ can be conveniently given in the form of a so-called \emph{consistency property}, which is just a set of finite sets  of sentences satisfying conditions which essentially code a winning strategy for Eloise in $\MEG(\phi)$. Sometimes it is more convenient to use a consistency property than  Model Existence Game. But as far as strategies of Eloise are concerned, the two are one and the same thing. Consistency properties have been successfully used to prove interpolation and preservations results in model theory, especially infinitary model theory \cite{MR0255383}.

Apart from first order and infinitary logic, the Model Existence Game can be used in the proof theory and model theory of  logic with generalized quantifiers and higher order logic. In the infinitary logic  $L_{\kappa\lambda}$, $\lambda>\omega$, the Model Existence Game yields so-called \emph{chain models}, rather than real models. In the case of generalized quantifiers, such as the quantifier ``there exists uncountably many", the Model Existence Game yields so-called weak models, which have to be transformed to real models by a model theoretic argument \cite{MR41:8217}. In the case of higher order logics the Model Existence Game yields so-called \emph{general models}, where the higher order variables range over a set of subsets and relations rather than over all subsets and relations \cite{MR36188}. The Model Existence Game can be used also in propositional and modal logic.

\section{The sentence is missing: The EF game}

In the EF game we have a model but no sentence. To indicate what kind of sentence we are looking for we have two models. The sentence is supposed to be true in one but false in the other. It may be that no such sentence can be found. Then the models are \emph{elementarily equivalent}. We describe a game in which strategies of one player track possibilities for elementary equivalence and the strategies of the other player track possibilities for a separating sentence.
In its simplest form this game, the EF-game (a.k.a. Ehrenfeucht-Fra\"\i ss\'e game, the back-and-forth game or the model comparison game), introduced by A. Ehrenfeucht \cite{MR23:A3666}, is the following: Suppose  $\ma$ and $\mb$ are two structures for the same vocabulary $L$.

The game $\ef_m(\ma,\mb)$ has two players Abelard and Eloise and $m$ moves. A position of the game is a  set 
\begin{equation}\label{position}
s=\{(a_0,b_0),\ldots, (a_{n-1},b_{n-1})\}
\end{equation} of pairs of elements such that  the $a_i$ are from $A$ and the $b_i$ are from $B$, and $n\le m$. In the beginning of the game the position is $\emptyset$. The rules are as follows: Suppose the position is $s$.
\begin{enumerate}
\item Abelard may choose some $a_n\in A$. Then Eloise chooses $b_n\in B$ and the next position is $s\cup\{(a_n,b_n)\}$.
\item Abelard may choose some $b_n\in B$. Then Eloise chooses $a_n\in A$ and the next position is $s\cup\{(a_n,b_n)\}$.
\end{enumerate}
Abelard wins if during the game the position  
(\ref{position}) is such that  $(a_0,\ldots,a_{n-1})$ satisfies some literal in $\ma$ but $(b_0,\ldots,b_{n-1})$ does not satisfy the corresponding literal in $\mb$. If Abelard does not win, then  Eloise is declared the winner. 

Intuitively, Eloise defends the proposition that $\ma$ and $\mb$ are very similar in the sense that isomorphic structures are similar (but $\ma$ and $\mb$ need not be actually isomorphic). Abelard doubts this similarity. Abelard expresses his doubt by asking where would the purported isomorphism map this or that element.

An extremely simple example is two  models of size $\ge m$ in the empty vocabulary. Then Eloise's winning strategy is based on her playing different elements when Abelard does the same. A slightly less trivial example is two finite linear orders of size $\ge 2^m$. 

Note that this game is determined. Moreover,
if Eloise  knows an isomorphism $f:\ma\to \mb$ she can respond by playing always so that $b_n=f(a_n)$.

\begin{figure}
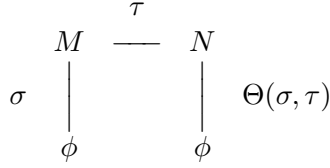

$$\begin{array}{ccccc}
&&\tau&&\\
&M&-\!\!\!-\!\!\!-&N&\\
\sigma&\bigg|&&\bigg|&\Theta(\sigma,\tau)\\
&\phi&&\phi
\end{array}$$
\caption{\label{fig1}Interaction of Evaluation Game and EF game.}\end{figure}

The following theorem demonstrates how the EF game and the Evaluation Game interact when we actually do have a sentence (see Figure~\ref{fig1}). The theorem is due to \cite{MR23:A3666}.

\begin{theorem} There is a function $\Theta$ such that if $\tau $ is a   strategy of Eloise in $\ef_m(\ma,\mb)$, 
 $\phi$ is an $L_{\infty\omega}$-sentence of quantifier rank $\le m$, and $\sigma$ is a  strategy of Eloise in $G(M,\phi)$, then $\Theta(\sigma,\tau)$ is a strategy of Eloise in  $G(N,\phi)$. If  $\tau$ and $\sigma$ are winning strategies, then so is $\Theta(\sigma,\tau)$.
\end{theorem}

\begin{proof}
We call a position of the game $\ef_m(\ma,\mb)$ a $\tau$-position if it arises while Eloise is playing $\tau$. Suppose $\phi$ is a sentence of quantifier rank $\le m$  and Eloise has a  strategy $\sigma$ in $G(M,\phi)$.  We call a position of the game $G(M,\phi)$ a $\sigma$-position, if it arises while Eloise is playing $\sigma$. We use $\tau$ and $\sigma$ to construct a  strategy of Eloise in  $G(N,\phi)$. 
 If the position of the game $G(N,\phi)$ is $(\psi,s)$, the strategy $\eta=\Theta(\sigma,\tau)$ of Eloise is to play simultaneously with $G(N,\phi)$ the game $\ef_m(\ma,\mb)$ and the game $G(M,\phi)$ and make sure that if $$\pi=\{(a_0,b_0),\ldots, (a_{n-1},b_{n-1})\}$$ is the current $\tau$-position in $\ef_m(\ma,\mb)$ and $s(x)=\pi(s'(x))$ for all $x$ in the domain of $s$, then $(\psi,s')$ is the current $\sigma$-position in $G(M,\phi)$ (see Figure~\ref{fig2}).

Let us check that it is possible for Eloise to play this strategy:

\begin{enumerate}

\item If the position is $(\psi,s)$ where $\psi$ is a literal, the game ends. 

\item If the position is $(\psi,s)$ where is $\psi$ is $\bigwedge_i\phi_i$, then Abelard chooses $i$ and  the next position is
$(\psi_i,s)$. Whichever he chooses, we let Abelard make the respective move $(\psi_i,s')$ in $G(M,\phi)$. 

\item If the position is $(\psi,s)$ where is $\psi$ is $\bigvee_i\phi_i$, then Eloise  chooses $i$ as follows. 
Since $(\psi,s')$ is a $\sigma$-position, the strategy $\sigma$ tells Eloise which of 
$(\psi_i,s')$ to play in $G(M,\phi)$. Then Eloise plays the respective $(\psi_i,s)$ in $G(N,\phi)$.

\begin{figure}[h!]
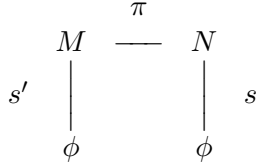

$$\begin{array}{ccccc}
&&\pi&&\\
&M&-\!\!\!-\!\!\!-&N&\\
s'&\bigg|&&\bigg|&s\\
&\phi&&\phi
\end{array}$$
\caption{The strategy $\Theta(\sigma,\tau)$\label{fig2},}\end{figure}
\item If the position  $(\psi,s)$ is $(\forall x\theta,s)$, then Abelard chooses $b_n\in N$ and the next position is
$(\theta,t)$, $t=s(b_n/x)$. We continue the game $\ef_m(\ma,\mb)$ from the $\tau$-position
$\{(a_0,b_0),\ldots, (a_{n-1},b_{n-1})\}$ letting Abe\-lard play $b_n$. The strategy $\tau$ tells Eloise to choose $a_n\in M$ so that 
\begin{equation}\label{ef1}
\pi'=\{(a_0,b_0),\ldots, (a_{n},b_{n})\}
\end{equation} is again a $\tau$-position. Now we continue the game $G(M,\phi)$ from position $(\forall x\theta,s')$ by letting Abelard play $a_n$. We reach the position $(\theta,t')$, $t'=s'(a_n/x)$, which is still a $\sigma$-position, and we have $\pi'(t'(y))=t(y)$ for all $y$ in the domain of $t'$.  
\item The position  $(\psi,s)$ is $(\exists x\theta,s)$. Now we continue the game $G(M,\phi)$ from position $(\exists x\theta,s')$ by letting Eloise play, according to $\sigma$, an element $a_n$ and we reach a new $\sigma$-position $(\theta,t')$, $t'=s'(a_n/x)$. We continue the game $\ef_m(\ma,\mb)$ from the $\tau$-position
$\{(a_0,b_0),\ldots, (a_{n-1},b_{n-1})\}$ letting Abe\-lard play $a_n$. The strategy $\tau$ tells Eloise to choose $b_n\in N$ so that (\ref{ef1}) is again a $\tau$-position.  We reach the position $(\theta,t)$, $t=s(b_n/x)$, and we have $t(y)=\pi(t'(y))$ for all $y$ in the domain of $t$.

\end{enumerate}

If $\sigma$ is a winning strategy and the game $G(N,\phi)$ ends in the position  $(\psi,s)$, where $\psi$ is a literal, then  Eloise wins because then $(\psi,s')$ is a $\sigma$-position meaning that $s'$ satisfies the literal $\psi$  in $M$, and $\tau$ being a winning strategy this means that $s$ satisfies the literal $\psi$ in $N$. 
\end{proof}

Just as the Model Existence Game can yield a model, the EF game can produce a sentence. The next theorem shows how this arises from co-operation of the EF game and the Evaluation Game (see Figure~\ref{fig3}). The theorem is due to \cite{MR23:A3666}. The emphasis on this kind of translations between strategies occurs already in \cite{MR3236216}.

\begin{figure}
$$\begin{array}{ccccc}
&&\tau&&\\
&M&-\!\!\!-\!\!\!-&N&\\
\Xi_E(\tau)&\bigg|&&\bigg|&\Xi_A(\tau)\\
&\phi&&\neg\phi
\end{array}$$
\caption{\label{fig3}}\end{figure}

\begin{theorem}\label{sent} There 
  is a sentence $\phi\in L_{\infty\omega}$ of quantifier rank $\le m$ and  mappings $\Xi_E$ and $\Xi_A$ such that if $\tau$ is a  strategy of Abelard in $\ef_m(M,N)$, then $\Xi_E(\tau)$ is a strategy of Eloise in $G(M,\phi)$, and   
   $\Xi_A(\tau)$ is a strategy of Abelard in  $G(N,\phi)$. If $\tau$ is a winning strategy, then $\Xi_E(\tau)$ and $\Xi_A(\tau)$ are winning strategies. If $L$ is finite and relational, the sentence $\phi$ is logically equivalent to a first order sentence of quantifier rank $\le m$.
\end{theorem}

\begin{proof}
Suppose $s$ is an assignment into $M$ with domain $\{x_0,\ldots,x_{n-1}\}$. Let
$$
\begin{array}{lcl}
\psi^{0,n}_{M,s}&=&\displaystyle{\bigwedge_{i}\psi_i}\\
\psi^{m+1,n}_{M,s}&=&(\forall x_n\displaystyle{\bigvee_{a\in M}}
\psi^{m,n+1}_{M,s(a/x_n)})\wedge(\displaystyle{\bigwedge_{a\in M}}\exists x_n\psi^{m,n+1}_{M,s(a/x_n)}),
\end{array}$$
where $\psi_i$ lists all the literals in the variables $x_0,\ldots,x_{n-1}$ satisfied by $s$ in $M$.

The sentence $\phi$ we need to prove the theorem is $\psi^{m,0}_{M,\emptyset}$.  
Note that this only depends on $M$. Clearly Eloise has a trivial strategy $\Xi_E(\tau)$ in $G(M,\phi)$ (independently of $\tau$), and this strategy is always a winning strategy: after Abelard has chosen a value for $x_n$, she uses that value as her choice of $a\in M$, and after Abelard has chosen $a\in M$, she uses that element as her value for $x_n$.
Note that the quantifier-rank of $\psi^{m,n}_{M,s}$ is always $m$.

We now describe the strategy $\Xi_A(\tau)$ of Abelard in $G(N,\phi)$. 
We call a position of the EF-game a \emph{$\tau$-position }if it arises while Abelard is playing $\tau$. 
Suppose $s$ is an assignment into $M$ and $s'$ an assignment into $N$, both with domain $\{x_0,\ldots,x_{n-1}\}$. 
We use $s\cdot s'$ to denote the set of pairs $(s(x_i),s'(x_i))$, $i=0,\ldots,n-1$. The strategy of Abelard is to play 
$G(N,\phi)$ in such a way that if the position at any point is $(\psi^{i,m-i}_{M,s},s')$, then $s\cdot s'$ is a $\tau$-position.
In the beginning the position is $(\phi,\emptyset)$.
\begin{enumerate}

\item Suppose the position in $G(N,\phi)$ is $(\psi^{i,m-i}_{M,s},s')$, $i>0$,
and the next move for Abelard in $\ef_m(M,N)$ according to $\tau$  is $a\in M$. The strategy of Abelard is to choose the latter conjunct of $\psi^{i,m-i}_{M,s}$. Then Abelard chooses the element $a\in M$ in the big conjunction move. Now it is the turn of Eloise to choose some $b\in N$ as the value of $x_{m-i}$ and that will be the next move of Eloise in $\ef_m(M,N)$. The new position  $s(a/x_{m-i})\cdot s'(b/x_{m-i})$ is still a $\tau$-position in $\ef_m(M,N)$. The next position in $G(N,\phi)$ is 
\begin{equation}\label{ef2}
(\psi^{i-1,m-i+1}_{M,s(a/x_{m-i})},s'(b/x_{m-i})).
\end{equation}

\item Suppose the position in $G(N,\phi)$ is $(\psi^{i,m-i}_{M,s},s')$, $i>0$,
and the next move for Abelard in $\ef_m(M,N)$ according to $\tau$  is $b\in N$. The strategy of Abelard is to choose the former conjunct where he plays $b$ as $x_{m-i}$. Now it is the turn of Eloise to choose some $a\in M$ in $G(N,\phi)$. The new position  $s(a/x_{m-i})\cdot s'(b/x_{m-i})$ is still a $\tau$-position in $\ef_m(M,N)$. The next position in $G(N,\phi)$ is (\ref{ef2}).

\item Finally the position is $(\psi^{0,m}_{M,s},s')$. Note that $s \cdot s'$ is still a $\tau$-position in $\ef_m(M,N)$. The game  $\ef_m(M,N)$ has now ended. Abelard now chooses the first (in some fixed enumeration) literal conjunct of 	the formula $\psi^{0,m}_{M,s}$ that  is not satisfied by $s'$ in $N$, if any exist, otherwise he simply chooses the first conjunct. 	

\end{enumerate}

Suppose now $\tau$ was a winning strategy of Abelard. Then at the end of the game $s \cdot s'$ is a winning position for Abelard and therefore he is indeed able to choose a conjunct of 	the formula $\psi^{0,m}_{M,s}$ that is not satisfied by  $s'$ in $N$. He has won $G(N,\phi)$. 	

If $L$ is finite and relational, all the conjunctions and disjunctions are essentially finite because there are only finitely many non-equivalent formulas of a fixed quantifier rank. Thus, although $\phi$ is, a priori, a formula of the infinitary logic $L_{\infty\omega}$, it is logically equivalent to a first order formula. 
\end{proof}

There is a tight connection between $\sigma$, $\tau$ and $\Xi(\sigma,\tau)$. This is reflected in a connection between $\phi$ and $\ef_m(M,N)$ which goes deeper into the structure of $\phi$ than the mere condition that its quantifier rank is at most $m$. If the non-logical symbols of $\phi$ are in $L'\subset L$, then it suffices that  $\tau$ is a strategy of Eloise  in the game $\ef_m(M\restriction L',N\restriction L')$ between the reducts $M\restriction L'$ and $N\restriction L'$. If we know more about the syntax of $\phi$, for example that it is existential, universal or positive, we can modify $\ef_m(M,N)$ accordingly by stipulating that Abelard only moves in $M$, only moves in $N$, or that he has to win by finding an atomic (rather than literal) relation which holds in $M$ but not in $N$. 

Likewise, many useful observations can be made about the sentence $\phi$ and the strategy 
$\tau$. Here are some. We  already observed that the quantifier rank of the separating sentence $\phi$ is the same as the length of the EF game for which Abelard has a winning strategy. If $\tau$ is a winning strategy of Abelard even in the game $\ef_m(M\restriction L',N\restriction L')$ for some $L'\subset L$, then the separating sentence $\phi$ can be chosen so that its non-logical symbols are all in $L'$. If $\tau$ is such that Abelard plays only in $M$, we can make $\phi$ existential. If $\tau$ is such that Abelard plays only in $N$, we can make $\phi$ universal. If Abelard wins with $\tau$ even the harder game in which   he has to win by finding an atomic (rather than literal) relation which holds in $M$ but not in $N$,  then we can take $\phi$ to be a positive sentence.

Strategies in $\ef_m(M,N)$ reflect structural properties of $M$ and $N$. If we know a strategy of Eloise in $\ef_m(M_i,N_i)$ for $i\in I$, we can construct strategies of Eloise for EF games between products and sums of the models $M_i$ and the respective products and sums of the models $N_i$. This can be extended to so-called $\kappa$-local functors \cite{MR50:12646}. For an example of the use of tree-decompositions, see e.g. \cite{MR2289798}.

Just as in the Evaluation Game, the EF game permits modifications which extend the above two theorems to other logics. Accordingly, EF games are known for infinitary logics, generalized quantifiers, and higher order logics. In modal logic the corresponding game is called a bisimulation game. Even propositional logic has an EF game although there are no quantifiers \cite{DBLP:books/daglib/p/HellaV15}. In the EF game for team semantics the players moves teams, not elements \cite{MR2351449}. 

The sentences arising from Theorem~\ref{sent} are behind the Distributive Normal Forms of J. Hintikka \cite{MR0069778} and are known in infinitary logic as building blocks of Scott Sentences \cite{MR34:32}.

In an important variant called the Pebble Game, the players can pick only a fixed finite number of elements from the models but they can give up some elements in order to make room for new ones. This game is closely related to first order (or infinitary) logic with a fixed finite number of variables (free or bound). It was introduced independently by N. Immerman and B. Poizat in 1982. A good reference is \cite{MR1167033}.

\section{Further work}

The Evaluation Game, the Model Existence Game and the EF game go so deep into the essential concepts of logic such as truth, consistency, and separating models by sentences, that a lot of research in logic can be represented in terms of these games. This may not be very interesting in itself. However, the translations of the strategies between the games suggest a coherent uniform approach to syntax and semantics and at the same time to model theory and proof theory. Moreover, both the Evaluation Game and the EF game are oblivious to whether the models are finite or infinite, which gives them a useful role in computer science logic. Despite the vast literature on each of the three games separately, there seems to be a lot of potential for the study of their interaction as a manifestation of the Strategic Balance of Logic.


\def\Dbar{\leavevmode\lower.6ex\hbox to 0pt{\hskip-.23ex \accent"16\hss}D}
  \def\cprime{$'$} \def\ocirc#1{\ifmmode\setbox0=\hbox{$#1$}\dimen0=\ht0
  \advance\dimen0 by1pt\rlap{\hbox to\wd0{\hss\raise\dimen0
  \hbox{\hskip.2em$\scriptscriptstyle\circ$}\hss}}#1\else {\accent"17 #1}\fi}

\end{document}